\documentclass[11pt]{amsart}

\usepackage[margin=1in]{geometry} 

\usepackage{amsmath,amsfonts,amsthm,amssymb}
\usepackage{mathtools}

\newtheorem{theorem}{Theorem}

\usepackage{hyperref}
\usepackage{ulem}

\newcommand{\WW}{\textcolor{blue}}
\usepackage{todonotes}

\newcommand{\norm}[2]{\Vert #1 \Vert_{#2}}
\newcommand{\subeqref}[2]{$ \eqref{#1}_{#2} $}
\newcommand{\mrm}[1]{\mathrm{#1}}

\include{frontmatter}

\title[2D Oldroyd-B]{2D incompressible inviscid Oldroyd-B equations: ill-posedness, long time existence, and high Weissenberg number limit}

\author[X. Liu]{Xin Liu}
\address{Department of Mathematics, Texas A\&M University, College Station, TX, USA}
\email{xliu23@tamu.edu}

\author[W. Wang]{Weinan Wang}
\address{Department of Mathematics, University of Oklahoma, Norman, OK, USA}
\email{ww@ou.edu}

\date{\today}

\begin{document}
\include{frontmatter}

\begin{abstract}
In this paper, we consider the high-Weissenberg number limit of a Voigt-regularized two-dimensional Oldroyd-B model for viscoelastic fluids. We first demonstrate that the Euler-Oldroyd-B system is both linearly and nonlinearly ill-posed in Sobolev spaces, exhibiting Hadamard instability. Then, we introduce a Voigt-type regularization on the stress tensor, which stabilizes the system. For the regularized model, we establish long time ($T \sim \mathcal O(\varepsilon^{-2/3})$) well-posedness and uniform energy estimates with respect to the relaxation parameter $\varepsilon>0$. Lastly, we prove that, as $\varepsilon \to 0$, the solutions converge to a solution of the 2-d incompressible Navier-Stokes equations over time intervals of size $\mathcal O(\varepsilon^{-2/3})$. The proof relies on a decomposition of the stress tensor, high-order energy estimates, and a detailed analysis of the nonlinear coupling terms. Our results provide a mathematical justification for the Newtonian limit of a regularized viscoelastic fluid model that is otherwise ill-posed.
\end{abstract}

\maketitle

\section{Introduction}

\subsection{The Euler-Oldroyd-B system}

The hydrodynamics of viscoelastic fluids are of fundamental importance in rheology and complex fluid dynamics. Among the various constitutive models describing such fluids, the Oldroyd-B model stands out as a prototypical example, capturing significant non-Newtonian phenomena such as stress relaxation and normal stress differences. The mathematical analysis of the Oldroyd--B system has attracted considerable attention; we refer to
\cite{Renardy2000,ConstantinSun2012} for background.
Global existence of weak solutions for several Oldroyd-type models (including regularized/diffusive variants)
was developed in the seminal works of Lions--Masmoudi \cite{LionsMasmoudi2000,LionsMasmoudi2001} and further
studied in \cite{FWW, Masmoudi2011,HuLelievre2007,BarrettBoyaval2011}. For more insights into the physics of viscoelastic
fluid flows, the reader is refered to \cite{Oswald2005,Ottinger2005,OttoTzavaras2008,LeBrisLelievre2009,WapperomHulsen1998}.
In the absence of stress diffusion, the constitutive equation is essentially hyperbolic/transport and the
global well-posedness theory for large data remains largely open.

In this paper, we investigate the dynamics of the two-dimensional incompressible Euler-Oldroyd-B system and its relaxation limit under a specific regularization. To be more precise, we are interested in the 2-d Oldroyd-B system in $\mathbb R^2$.
Let $ u : \mathbb T^2 \times [0,\infty) \mapsto \mathbb R^2 $ be the incompressible velocity field. Denote by $ p $ the flow pressure. $ \tau: \mathbb T^2 \times [0,\infty) \mapsto \mathbb R^4 $ is viscous stress tensor of the flow. We write the symmetric and anti-symmetric gradients of $ u $ as
\begin{align}
    \label{def:sym-gradiant}
    D(u) &:= \frac{\nabla u + \nabla u^\top}{2}, \\
\intertext{and}
    \label{def:skew-gradient}
    W(u) &: =\frac{\nabla u - (\nabla u)^T}{2}.
\end{align}
Then the 2D Euler-Oldroyd-B system can be written as
\begin{equation}
\label{sys:oldroyd-b}
    \begin{cases}
        \partial_t u + u\cdot \nabla u +\nabla p=\nabla \cdot \tau, \\
        \partial_t \tau + u\cdot \nabla \tau + Q(\tau, \nabla u) =\frac{1}{\epsilon^2} (D(u)-\tau), \\
        \nabla \cdot u=0,
    \end{cases}
\end{equation}
where the {rotation correction} $ Q $ is given by
\begin{equation}
\label{def:Q}
        Q(\tau, \nabla u) := \tau W(u)-W(u)\tau + b(\tau D(u)+D(u)\tau).
\end{equation}

Unfortunately, as we will show later, system \eqref{sys:oldroyd-b} is in general ill-posed in all Sobolev spaces (Theorem \ref{thm:ill-posedness}). This is due to the loss of regularity brought by $ Q(\tau, \nabla u) $ in \subeqref{sys:oldroyd-b}{2}. For this reason, we consider the following Voigt regularization:
\begin{equation}
\label{sys:oldroyd-bb}
    \begin{cases}
        \partial_t u + u\cdot \nabla u +\nabla p=\nabla \cdot \tau,\\
        \partial_t (\tau - \Delta \tau  ) + u\cdot \nabla \tau +  Q(\tau, \nabla u) =\frac{1}{\epsilon^2} (D(u)-\tau),\\
        \nabla \cdot u=0.
    \end{cases}
\end{equation}

\smallskip

In our nondimensionalization the relaxation time scale is $\varepsilon^2$, i.e.\ the Weissenberg number is
$\mathrm{We}\sim \varepsilon^{-2}$. Hence the high-Weissenberg/fast-relaxation regime corresponds to
$\varepsilon\to 0$, in which the constitutive law enforces $\tau \approx D(u)$ and the system is expected to
converge to the incompressible Navier--Stokes equations.

To capture the large time dynamics of the limiting 2-d Navier-Stokes equations, we consider the following tightened variable:
\begin{equation}
    \label{def:tight-variable}
    \sigma := \tau - D(u). 
\end{equation}
Then one can write system \eqref{sys:oldroyd-bb} in terms of $ (u,\sigma) $ as follows:
\begin{equation}
\label{sys:tightened-oldroyd}
    \begin{cases}
        \partial_t u + u\cdot \nabla u +\nabla p= \Delta u+ \nabla \cdot \sigma,\\
        \partial_t (\sigma -  \Delta \sigma) + u\cdot \nabla \sigma + u\cdot \nabla D(u) +  Q(\sigma + D(u), \nabla u)  \\
        \qquad  + \partial_t D(u) -  \partial_t \Delta D(u) = -\frac{1}{\epsilon^2} \sigma,\\
        \nabla \cdot u=0.
    \end{cases}
\end{equation}
We refer to \eqref{sys:tightened-oldroyd} as the tightened system. 

\smallskip

A central contribution of this work is to demonstrate that the system \eqref{sys:oldroyd-b}, in its inviscid form, exhibits severe pathological behavior. In section \ref{sec:ill-posedness}, we prove that the system is linearly ill-posed in Sobolev spaces. Furthermore, we establish nonlinear ill-posedness in the sense of Hadamard. This suggests that the standard Euler-Oldroyd-B model, without additional dissipative mechanisms, is mathematically physically inconsistent in certain regimes.

Motivated by this ill-posedness and the need for robust numerical schemes, we investigate the \textit{Voigt regularization} of the stress equation. Voigt regularization has been successfully applied to the Navier-Stokes and Euler equations to mitigate regularity issues. 

The primary goal of this paper is to analyze the relaxation limit of the regularized system as $\varepsilon \to 0$. In this limit, formally, the constitutive equation forces $\tau \approx D(u)$, and the system is expected to converge to the incompressible Navier-Stokes equations. This connects the complex visco-elastic dynamics to the classical Newtonian fluid limit.
The Newtonian limit for weakly viscoelastic flows has a long history; see
Saut \cite{Saut1986,Saut2012} and the rigorous analysis in Molinet--Talhouk \cite{MolinetTalhouk2008}. There are also results using relative entropies for the study of the long-time behavior of some micro-macro
models for dilute solutions of polymers \cite{JourdainLeBrisLelievreOtto2006} and relative entropy methods are also the key to the estimates of Otto and Tzavaras in \cite{OttoTzavaras2008}. See also \cite{LattanzioTzavaras2006, LattanzioTzavaras2013}. In the present work we establish such a limit for the Voigt-regularized model on a long
time window that scales like $\mathcal O(\varepsilon^{-2/3})$, with estimates uniform in $\varepsilon$.

From a broader perspective, the passage $\tau\to D(u)$ can be viewed as a {relaxation mechanism} toward an
incompressible viscous flow, which is reminiscent of relaxation approximations to incompressible fluid models;
see, for example, \cite{BrenierNataliniPuel2004,NataliniRousset2006} and related singular-limit works such as
\cite{Brenier2000}. Compared with classical Newtonian-limit analyses for Oldroyd-type systems (e.g. Bresch--Prange \cite{bresch}), our Voigt-regularized/tightened formulation yields a velocity equation with a nondegenerate Laplacian \(\Delta u\), so the low-order coercive estimates do not rely on a small prefactor such as \(1-\omega\).

\subsection{Main contribution and highlights}
\subsubsection{Two-tier energy perspective. }
A important tool of our a priori analysis is a {two-tier} energy scheme, inspired by the framework developed by Guo and Tice in their works \cite{Guo2011, guo2013almost,guo2013decay} on viscous surface waves without surface tension. 
In that setting, one couples (i) a {high-order} energy that controls the full regularity of the solution but whose top derivatives cannot be closed by dissipation alone, with (ii) a {low-order} energy that enjoys genuine decay and provides the time-integrability needed to control the delicate highest-order terms.
We follow the same philosophy here: we introduce a high-tier energy to propagate regularity, together with a lower-tier energy designed to decay.
The decay of $\mathcal E_{\mathrm{low}}$ provides the quantitative smallness/time-integrability that closes the nonlinear estimates at the top level, while the boundedness of $\mathcal E_{\mathrm{high}}$ allows us to upgrade decay rates once the lower tier is controlled.

\subsubsection{Ill-posedness of the inviscid Euler--Oldroyd--B system in Sobolev spaces.}
A second main contribution of this work is to identify a {Hadamard-type instability} for the inviscid Euler--Oldroyd--B dynamics \eqref{sys:oldroyd-b} (for any fixed relaxation parameter $\varepsilon>0$): the problem is ill-posed in every Sobolev space $H^s$, $s\ge 3$, in the sense that the data-to-solution map fails to be continuous.
More precisely, we linearize \eqref{sys:oldroyd-b} around the nontrivial constant stress state $(\bar u,\bar\tau)=(0,\tau_a)$ with $\tau_a=a\,\mathrm{Id}$, and exhibit explicit Fourier modes whose growth rate scales like
\[
Re \lambda_+(k)\sim c\,|k|\qquad\text{as }|k|\to\infty
\]
(for instance at $a=-2$), i.e.\ arbitrarily high frequencies amplify on arbitrarily short time scales.
This high-frequency growth implies that the linearized solution operator cannot be bounded on $H^s$, which already rules out Hadamard well-posedness.
We then upgrade this mechanism to the nonlinear system by a scaling/compactness contradiction argument: assuming a Sobolev well-posedness estimate, one can pass to a limiting problem that inherits the same unstable growth, leading to norm inflation and hence failure of continuous dependence.
This phenomenon is distinct from recent ill-posedness results for {viscous} Oldroyd--B-type systems in critical Besov classes (e.g. \cite{Li_2025}), where the velocity equation contains $-\Delta u$ and the stress evolves by transport/rotation.
Here the instability occurs already in the inviscid setting and persists {despite} the presence of the relaxation term $\varepsilon^{-2}(D(u)-\tau)$: the culprit is the first-order coupling through the rotation correction $Q(\tau,\nabla u)$, which enables high-frequency amplification around a constant-stress background.
To the best of our knowledge, this Sobolev-space Hadamard instability for the inviscid Euler--Oldroyd--B system (in the formulation \eqref{sys:oldroyd-b}) has not been previously isolated in the literature, and it provides a clear mathematical obstruction motivating the Voigt-type regularization \eqref{sys:oldroyd-bb} studied in the rest of the paper.

\subsubsection{Voigt regularization}
A natural way to control smoothing-driven instability is to introduce a {Voigt (Kelvin--Voigt) type}
regularization, i.e. to modify the time evolution by applying an elliptic or fractional operator to the
time derivative (schematically, replacing $\partial_t$ by $(I+\varepsilon \mathcal{A})\partial_t$ for a positive
self-adjoint elliptic $\mathcal{A}$). In incompressible fluid dynamics, the prototypical Navier--Stokes--Voigt
model \cite{KT} adds a term of the form $-\alpha^2 \Delta \partial_t u$, and can be interpreted as a Kelvin--Voigt
viscoelastic model; it has been used both as a physically meaningful viscoelastic closure and as a mathematically
favorable regularization that improves well-posedness and long-time behavior.
A notable feature of Voigt regularizations (in contrast to adding viscosity) is that they are typically
{conservative}: they regularize the conserved energy, preserve steady states, and avoid spurious boundary
layers, which makes them well suited for long-time dynamics and limiting procedures.
Beyond fluids, Voigt regularizations have also been exploited in magnetohydrodynamic contexts \cite{CP, LT} (e.g. as a magnetic
relaxation mechanism to construct equilibria as long-time limits) and in Boussinesq-type systems \cite{I}, where they yield
global regularity and convergence as the regularization is removed. 
Similarly, in the recent study of sea-ice dynamics, the ill-posedness of the original sea-ice rheology \cite{boutros-2} also leads to the study of the non-Newtonian stress tensor with Voigt regularization \cite{boutros-1}. 
These precedents motivate our use of Voigt/Kelvin--Voigt ideas as a principled regularization framework in the
present inverse setting.

\subsection{Main theorems}

Without loss of generality, we investigate only the case when
\begin{equation}
    \label{0-momentum}
    \int u \,dx = 0,
\end{equation}
which is an invariance and allows us to apply the Poincar\'e inequality:
\begin{equation}
    \label{poincare}
    \norm{u}{L^2} \lesssim \norm{\nabla u}{L^2}. 
\end{equation}

\begin{theorem}[Nonlinear ill-posedness]
\label{thm:ill-posedness}
System \eqref{sys:oldroyd-b} with arbitrary fixed $ \varepsilon $ is ill-posed in any Sobolev space $ H^s $, $ s \geq 3 $, in the sense of Hadamard. 
\end{theorem}

The proof of linear ill-posedness of system \eqref{sys:oldroyd-b} is done in section \ref{subsec:linear-ill-posedness}. The nonlinear ill-posedness is done in section \ref{subsec:nonlinear-ill-posedness}.

\begin{theorem}[Long time existence]
\label{thm:long-time}
Consider initial data $ (u(0), \tau(0)) = (u_\mrm{in}, \tau_\mrm{in}) \in H^6(\mathbb T^2) $, satisfying
\begin{equation}
    \label{average-zero}
    \int_{\mathbb T^2} u_\mrm{in} \,dx = 0. 
\end{equation}
There exists $ \varepsilon_0 \in (0,1) $, depending on $ \norm{u_\mrm{in}, \tau_\mrm{in}}{H^6} $, such that for $ 0< \varepsilon <  \varepsilon_0 $, there exists a unique classical solution to the 2D Euler-Oldroyd-B system with Voigt regularization, i.e., system \eqref{sys:oldroyd-bb} for 
\begin{equation}
\label{thm-1-time}
    0 \leq  t \leq \frac{1}{\varepsilon^{2/3}},
\end{equation} 
satisfying that
\begin{equation}
    \label{thm-1-est}
    \begin{gathered}
    \sup_{0 \leq t \leq \varepsilon^{-2/3}} (\norm{u(t),\varepsilon \tau(t), \varepsilon\nabla\tau(s)}{H^6}^2 + \norm{\tau(t) - D(u)(t), \nabla (\tau(t) - D(u)(t))}{H^2}^2 ) \\
    + \int_0^{\varepsilon^{-2/3}} (\norm{\tau(t)}{H^6}^2 + \norm{\nabla u(t), \frac{\tau(t) - D(u)(t)}{\varepsilon}}{H^2}^2 ) \,dt < C,
    \end{gathered}
\end{equation}
where $ C \in (0,\infty) $ is a finite constant depending only on the initial data. 
\end{theorem}

    The proof of theorem \ref{thm:long-time} is done in section \ref{sec:long-time}, which is based on the \textit{a priori} estimates in sections \ref{sec:low-norm-est} and \ref{sec:high-norm-est}.

\begin{theorem}[High Weissenberg number limit] 
\label{thm:hw-limit} Under the same assumptions as in theorem \ref{thm:long-time}, for arbitrary $ T \in (0,\infty) $, there exists $ 0 < \varepsilon_T \leq \varepsilon_0 $, such that for $ \varepsilon < \varepsilon_T $ and $ \varepsilon \rightarrow 0^+ $, one has that, there exists $ v \in L^\infty (0,T;H^6(\mathbb T^2) $ such that
\begin{align}
    u & \rightarrow  v && \text{in} \ C(0,T;H^5(\mathbb T^2)), \\
    \tau & \rightarrow D(v) && \text{in} \ L^2(0,T;H^{2}(\mathbb T^2)).
\end{align}
Moreover, $ v $ satisfies the Navier-Stokes equations, i.e., 
\begin{equation}
    \partial_t v + v \cdot \nabla v + \nabla \pi = \nabla \cdot D(v), \qquad \nabla \cdot v = 0. 
\end{equation}
\end{theorem}
\begin{proof}
    The proof of theorem \ref{thm:hw-limit} is done in section \ref{sec:hw-limit}.
\end{proof}

\section{Ill-posedness of the first order Oldroyd-b system \eqref{sys:oldroyd-b}}
\label{sec:ill-posedness}

\subsection{Linear ill-posedness}
\label{subsec:linear-ill-posedness}
In this subsection, we show that \eqref{sys:oldroyd-b} is linearly ill-posed in any Sobolev space, for any fixed $ \varepsilon $. Without loss of generality, consider the following system:
\begin{equation}
\label{sys:oldroyd-b-1}
    \begin{cases}
        \partial_t u + u\cdot \nabla u +\nabla p=\nabla \cdot \tau,\\
        \partial_t \tau + u\cdot \nabla \tau + Q(\tau, \nabla u) = D(u)-\tau, \\
        \nabla \cdot u=0.
    \end{cases}
\end{equation}
Let \begin{equation}\label{def:eta} \eta := \tau - \tau_a ,\end{equation}
where $ \tau_a $ is a constant tensor given by, for arbitrary $ a \in \mathbb R $, 
\begin{equation}
    \tau_a := a \begin{pmatrix}
        1 & 0 \\ 0 & 1
    \end{pmatrix}.
\end{equation}
Then $ (u,\eta) $ satisfies the following system:
\begin{equation}
\label{sys:oldroyd-b-12}
    \begin{cases}
        \partial_t u + u\cdot \nabla u +\nabla p=\nabla \cdot \eta,\\
        \partial_t \eta + u\cdot \nabla \eta + Q(\eta, \nabla u) + Q(\tau_a,\nabla u)= D(u)-\eta - \tau_a, \\
        \nabla \cdot u=0.
    \end{cases}
\end{equation}
The linearization of system \eqref{sys:oldroyd-b-12} around $ (\overline u = 0, \overline \eta = 0) $ is given by, using $ (u_l, \eta_l) $ to denote the linearized variable,
\begin{equation}
    \label{sys:linear-o-b-a}
    \begin{cases}
        \partial_t u_l + \nabla p_l = \nabla \cdot \eta_l, \\
        \partial_t \eta_l + Q( \tau_a, \nabla u_l) = D(u_l) - \eta_l-\tau_a, \\
        \nabla \cdot u_l = 0.
    \end{cases}
\end{equation}

We are looking for solutions to \eqref{sys:linear-o-b-a} with the form
\begin{equation}
\label{linear:ansatz}
u_l := (0, 1)^\top e^{\lambda t} e^{ikx}, \qquad p_l := Pe^{\lambda t}e^{ikx}, \qquad 
    \text{and} \quad \eta_l := (e^{-t} - 1) \tau_a  + e^{\lambda t}\Sigma e^{ikx}, 
\end{equation}
where $ \lambda \in \mathbb C $, $ k \in \mathbb Z $, $x \in \mathbb R $, $ P \in \mathbb C $, and $ \Sigma \in \mathbb C^{2\times 2} $, to be determined.

Direct calculation yields that
\begin{equation}
    \nabla u = \begin{pmatrix}
        0 & 0\\
        ik&0
    \end{pmatrix} e^{\lambda t} e^{ikx},
\end{equation}
therefore, 
\begin{equation}
    D(u)=\frac{\nabla u + (\nabla u)^T}{2}=\begin{pmatrix}
        0 & ik/2\\
        ik/2&0
    \end{pmatrix} e^{\lambda t} e^{ikx},
\end{equation}
and
\begin{equation}
    W(u)=\frac{\nabla u - (\nabla u)^T}{2}=\begin{pmatrix}
        0 & -ik/2\\
        ik/2&0
    \end{pmatrix} e^{\lambda t} e^{ikx}.
\end{equation}
Hence, after substituting \eqref{linear:ansatz} into \eqref{sys:linear-o-b-a}, one has that
\begin{align}
    i k P & = ik \Sigma_{11}, \\
    \lambda & = ik \Sigma_{21}, \\
    (\lambda + 1) \Sigma & = \begin{pmatrix}
         0&  ik(a+\frac{1}{2}) \\  ik(a+\frac{1}{2}) & 0
    \end{pmatrix}.
\end{align}
Solving the above system of equations leads to, for $ k\neq 0 $,
\begin{gather}
    \Sigma_{11} = \Sigma_{22} = P = 0, \\
    \Sigma_{21} = \Sigma_{12} = -i\frac{\lambda}{k},\\
    \label{eigen-value}
    \lambda^2 + \lambda + (a + \frac{1}{2}) k^2 = 0.
\end{gather}
That is,
\begin{equation}
    \lambda_\pm := \frac{-1\pm \sqrt{1-(4a+2)k^2}}{2}.
\end{equation}

In particular, for $ a = - 2 $, we find the following growing mode for $  k\neq 0 $:
\begin{align}
\label{unstable-eigenvalue}
    \lambda_+ (k) & = \frac{-1 + \sqrt{1+6k^2}}{2} \simeq \frac{\sqrt{6}}{2} k \geq k \qquad \text{for} \ \vert k\vert  \rightarrow \infty, \\
    \label{unstable-sigma}
    \Sigma(k) = \Sigma_+(k) & = \begin{pmatrix}
       0 & -i \frac{(-1 + \sqrt{1+6k^2})}{2k} \\ -i \frac{(-1 + \sqrt{1+6k^2})}{2k} & 0
    \end{pmatrix}, \qquad P = 0.
\end{align}
This implies the linear instability in the sense of Hadamard.

\subsection{Nonliear ill-posedness}\label{subsec:nonlinear-ill-posedness} To show the nonlinear ill-posedness of \eqref{sys:oldroyd-b-1},  we employ similar arguments as in \cite{Guo2011} for the Rayleigh-Taylor instability. However, in contrast to \cite{Guo2011}, the instability of system \eqref{sys:oldroyd-b-1} occurs at the nontrivial state, for some $ a \neq 0 $,
\begin{equation}
    \overline u = 0, \qquad \overline\tau = \tau_a. 
\end{equation}
To accommodate this fact, inspired by the linear analysis in section \ref{subsec:linear-ill-posedness}, let
\begin{equation}
    \rho:= \tau - e^{-t} \tau_a.
\end{equation}
Then system \eqref{sys:oldroyd-b-1} can be written as the following system of $ (u, \rho) $:
\begin{equation}
    \label{sys:nonlinear-illposedness}
    \begin{cases}
        \partial_t u + u\cdot \nabla u +\nabla p=\nabla \cdot \rho,\\
        \partial_t \rho + u\cdot \nabla \rho + Q(\rho, \nabla u) + Q(e^{-t}\tau_a,\nabla u) = D(u)-\rho, \\
        \nabla \cdot u=0.
    \end{cases}
\end{equation}
It is remained to show that the nonlinear system \eqref{sys:nonlinear-illposedness} is ill-posed. This is shown by contradiction. 

For any $ s \geq 3 $, 
assume that system \eqref{sys:nonlinear-illposedness} is well-posed in $ (u,\rho) \in L^\infty(0,T; H^s) $ for some $ T \in (0,\infty) $, and the following estimate holds:
\begin{equation}
    \label{est:assume-ill-posedness}
    \sup_{0\leq t \leq T} \norm{u(t),\rho(t)}{H^s} + \sup_{0\leq t \leq T}\norm{\partial_t u(t), \partial_t \sigma(t)}{H^{s-1}} \leq H(\norm{u_\mrm{in},\rho_\mrm{in}}{H^s}) \norm{u_\mrm{in},\rho_\mrm{in}}{H^s},
\end{equation}
where
$
    (u_\mrm{in}, \rho_\mrm{in}) = (u,\rho)\vert_{t=0}
$, $ H $ is a non-decreasing continuous function, $ H(0) < C $, and $ T $ is non-increasing with respect to $ \norm{u_\mrm{in},\rho_\mrm{in}}{H^s} $. 
For $ \delta \in (0,1) $, Consider
\begin{equation}
    \label{ill-posed:initial}
    (u_\mrm{in},\rho_\mrm{in}) = (\delta \underline u_\mrm{in},\delta \underline \rho_\mrm{in}),
\end{equation}
and
let
\begin{equation}
    \label{ill-posed:scaling}
    u_\delta: = \frac{1}{\delta} u , \qquad \rho_\delta:= \frac{1}{\delta} \rho, \qquad\text{and} \quad  p_\delta := \frac{1}{\delta}p.
\end{equation}
Then from \eqref{sys:nonlinear-illposedness}, one has that
\begin{equation}
\label{sys:scaled-illposed}
    \begin{cases}
        \partial_t u_\delta + \delta u_\delta\cdot \nabla u_\delta + \nabla p_\delta=\nabla \cdot \rho_\delta,\\
        \partial_t \rho_\delta + \delta u_\delta\cdot \nabla \rho_\delta + \delta Q (\rho_\delta, \nabla u_\delta) + Q(e^{-t}\tau_a,\nabla u_\delta) = D (u_\delta) -  \rho_\delta, \\
        \nabla \cdot u_\delta=0.
    \end{cases}
\end{equation}
Thanks to \eqref{est:assume-ill-posedness}, one can conclude that, without loss of generality, there exists $ \underline T > 1 $ and $ C \in (0,\infty) $, such that
\begin{equation}
    \sup_{0\leq t \leq \underline T}( \norm{u_\delta(t),\rho_\delta(t)}{H^s} + \norm{\partial_t u_\delta(t),\partial_t \rho_\delta(t)}{H^{s-1}} ) \leq C \norm{\underline u_\mrm{in},\underline \rho_\mrm{in}}{H^s}.
\end{equation}
Therefore, there exists $ (\underline u, \underline \rho) \in L^\infty(0,\underline T;H^s)\cap C(0,\underline T; H^{s-1}) $ such that, as $ \delta \rightarrow 0 $, by choosing a subsequence, 
\begin{equation}
    \begin{aligned}
        (u_\delta, \rho_\delta) & \overset{*}{\rightharpoonup} (\underline u,\underline \rho) && \text{in} ~ L^\infty(0,T;H^s), \\
        (u_\delta, \rho_\delta) & \rightarrow (\underline u,\underline \rho) && \text{in} ~ C(0,T;H^{s-1}).
    \end{aligned}
\end{equation}
Then it is easy to verify from \eqref{sys:scaled-illposed} that
\begin{equation}
    \label{sys:nonlinear-illposed-2}
    \begin{cases}
        \partial_t \underline u + \nabla \underline p = \nabla \cdot \underline \rho, \\
        \partial_t \underline \rho + Q(e^{-t} \tau_a, \nabla \underline u) = D(\underline u) - \underline\rho,
    \end{cases}
\end{equation} 
and 
\begin{equation}
    \label{est:nonlinear-ill-posedness-2}
    \sup_{0\leq t \leq 1} \norm{\underline u(t),\underline \rho(t)}{H^s} \leq C \norm{\underline u_\mrm{in},\underline \rho_\mrm{in}}{H^s}.
\end{equation}
System \eqref{sys:nonlinear-illposed-2} is very similar to system \eqref{sys:linear-o-b-a}. Inspired by the linear analysis, we look for solutions to system \eqref{sys:nonlinear-illposed-2} of the following form: for $ k \in \mathbb Z $ and $ x \in \mathbb R $,
\begin{equation}
    \label{ansatz_nonlinear}
    \underline u: = (0,1)^\top \alpha(t) e^{ikx}, \qquad \underline p := \beta(t) \pi e^{ikx}, \qquad \text{and} \quad \underline \rho := \beta(t) \zeta e^{ikx}, 
\end{equation}
for some $ \alpha, \beta \in C(0,1) $, $ \pi \in \mathbb C $, and $ \zeta \in \mathbb C^{2\times 2} $. Then one has that, repeating the arguments as in section \ref{subsec:linear-ill-posedness},
\begin{align}
    ik \pi & = ik \zeta_{11},\\
    \alpha'(t) & = \beta(t) ik \zeta_{21}, \\
    (\beta'(t) + \beta(t)) \zeta &= \begin{pmatrix}
        {0} & ik(e^{-t}a + \frac{1}{2})\alpha(t) \\ ik(e^{-t}a + \frac{1}{2})\alpha(t) &{0}
    \end{pmatrix}.
\end{align}
That is, we have 
\begin{equation}
\label{ode-001}
    \pi = \zeta_{11} = \zeta_{22} = 0 ,\qquad \alpha'(t) = ik \zeta_{21} \beta(t), \qquad (\beta'(t)+ \beta(t)) \zeta_{21} = ik (e^{-t}a + \frac{1}{2}) \alpha(t).
\end{equation}

Now we are ready to choose the initial date. To simplify the presentation, we only consider $ k > 1 $. Let 
\begin{equation}
\label{initial-data-nonlinear-illposed}
    a = -2 e ,\qquad \zeta_{21} = \zeta_{12} = - i, \qquad \alpha(0) = \frac{2}{k^s}, \qquad \beta(0) = \frac{2\lambda_+(k)}{k^{s+1}},
\end{equation}
where $ \lambda_+ (k) $ is given by \eqref{unstable-eigenvalue}. 
Then it is easy to verify that 
\begin{equation}
    \label{initial-data-nonlinear-illposed-2}
    \norm{\underline u_\mrm{in},\underline\rho_\mrm{in}}{H^s} < C < \infty, \qquad \text{independent of $ k$},
\end{equation}
and 
\begin{equation}
    \alpha'(t) = k \beta(t), \qquad \beta'(t) + \beta(t) =  k ( 2 e^{1-t}  - \frac{1}{2})\alpha(t).
\end{equation}
Notice that $ 2 e^{1-t} - \frac{1}{2} \geq \frac{3}{2}$ for $ 0 \leq t \leq 1 $ and it is easy to verify that for $ 0 \leq t \leq 1 $, 
\begin{equation}
    \alpha(t) > 0, \quad \beta(t) > 0,
\end{equation}
and
\begin{equation}
    \alpha'(t) = k \beta(t), \qquad \beta'(t) + \beta(t) \geq \frac{3}{2}k \alpha(t).
\end{equation}

Consider the following initial value problem:
\begin{equation}
    \alpha_1'(t) = k \beta_1(t), \quad \beta_1'(t) + \beta_1(t) = \frac{3}{2}k \alpha_1(t), \quad \alpha_1(0) = \frac{1}{k^s} < \alpha(0), \quad \beta_1(0) = \frac{\lambda_+(k)}{k^{s+1}} < \beta(0).
\end{equation}
Then,
\begin{equation}
    \alpha_1(t) = \frac{1}{k^3}e^{\lambda_+(k) t}, \qquad \beta_1(t) = \frac{\lambda_+(k)}{k^4} e^{\lambda_+(k) t}.
\end{equation}
Moreover, by comparison principle, one can verify that for $ 0 \leq t \leq 1 $ and large $ k $,
\begin{equation}
    \alpha(t) > \alpha_1(t) \geq \frac{e^{kt}}{k^s}, \qquad \beta(t) > \beta_1(t) \geq \frac{e^{kt}}{k^{s}}.
\end{equation}
This implies that 
\begin{equation}
    \norm{\underline u(1),\underline \rho(1)}{H^s} \gtrsim  e^{k} \rightarrow \infty \qquad \text{as} ~ k \rightarrow \infty,
\end{equation}
which is contradictory to \eqref{est:nonlinear-ill-posedness-2}. 
This finishes the proof of theorem \ref{thm:ill-posedness}.

\section{Low norm estimate for the tightened variables}
\label{sec:low-norm-est}

\subsection*{$L^2$ estimate}
Taking the $ L^2 $-inner product of \subeqref{sys:tightened-oldroyd}{1} with $ u $ and applying integration by parts in the resultant lead to 
\begin{equation}
\label{low-ene:001}
\frac{1}{2} \frac{d}{d t}\|u\|_{L^2}^2+\|\nabla u\|_{L^2}^2=-\int \sigma: D(u) \,  d x \leq \frac{1}{2} \norm{\nabla u}{L^2}^2 + \norm{\sigma}{L^2}^2. 
\end{equation}
Meanwhile, taking the $L^2$-inner product of \subeqref{sys:tightened-oldroyd}{2} with $\sigma$ and applying integration by parts yield
\begin{equation}
\label{low-ene:002}
\begin{gathered}
\frac{1}{2}\frac{d}{dt}\|\sigma, \nabla \sigma \|_{L^2}^2 + \frac{1}{\varepsilon^2}\|\sigma\|_{L^2}^2 
= - \int_{\mathbb{R}^2} ( \partial_t D(u) - \partial_t \Delta D(u) ):\sigma\,dx \\  
 - \int u \cdot \nabla D(u) : \sigma \,dx  - \int Q(\sigma + D(u), \nabla u) : \sigma\,dx \leq \frac{1}{2\varepsilon^2} \norm{\sigma}{L^2}^2 \\
 + \varepsilon^2 C (\norm{\partial_t D(u) - \partial_t \Delta D(u), u\cdot \nabla D(u),Q(\sigma+D(u),\nabla u)}{L^2}^2).
\end{gathered}
\end{equation}
Therefore, integrating \eqref{low-ene:001} and \eqref{low-ene:002} in time yields
\begin{equation}
    \label{low-ene:003}
    \begin{gathered}
    \sup_{0\leq s \leq t} \norm{u(s),\sigma(s), \nabla \sigma(s)}{L^2}^2 + \int_0^t \norm{\nabla u(s), \frac{\sigma(s)}{\varepsilon}}{L^2}^2 \,ds \lesssim \norm{u_\mrm{in},\sigma_\mrm{in}}{L^2}^2 \\ + \varepsilon^2 \underbrace{\int_0^t \norm{\frac{\sigma}{\varepsilon},\partial_t D(u) - \partial_t \Delta D(u), u\cdot \nabla D(u),Q(\sigma+D(u),\nabla u)}{L^2}^2 \,ds}_{=:\mathcal N_1(t)}. 
    \end{gathered}
\end{equation}

\subsection*{$H^1$ estimate} Let the vorticity of $ u $ be
\begin{equation}
    \label{low-ene:004}
    \omega:= \nabla \times u. 
\end{equation}
Then from \subeqref{sys:tightened-oldroyd}{1}, since we are in the two spatial dimensions. one has that
\begin{equation}
    \label{low-ene:005}
    \partial_t \omega + u \cdot \nabla \omega = \Delta \omega + \nabla \times \nabla \cdot \sigma. 
\end{equation}
Taking the $ L^2 $-inner product of \eqref{low-ene:005} with $ \omega $ and applying integration by parts lead to 
\begin{equation}
    \label{low-ene:006}
    \dfrac{1}{2} \dfrac{d}{dt} \norm{\omega}{L^2}^2 + \norm{\nabla \omega}{L^2}^2 = - \int (\nabla\cdot \sigma) \cdot \nabla^\top \omega\,dx \leq \frac{1}{2}\norm{\nabla \omega}{L^2}^2 + \norm{\nabla \sigma}{L^2}^2. 
\end{equation}
Similarly, 
applying $\partial \in \lbrace \partial_x, \partial_y \rbrace $ to \subeqref{sys:tightened-oldroyd}{2} yields
\begin{equation}
    \label{low-ene:007}
    \begin{gathered}
    \partial_t (\partial \sigma - \Delta \partial \sigma) + u\cdot \nabla \partial \sigma + \partial u \cdot \nabla \sigma + \frac{1}{\varepsilon^2} \partial \sigma \\
    = - \partial_t D (\partial u) + \partial_t \Delta D(\partial u) - \partial \lbrack u\cdot \nabla D(u)\rbrack - \partial\lbrack Q(\sigma+D(u),\nabla u) \rbrack.
    \end{gathered}
\end{equation}
Taking the $ L^2 $-inner product of \eqref{low-ene:007} with $ \partial \sigma $ and applying integration by parts yield
\begin{equation}
    \label{low-ene:008}
    \begin{gathered}
        \dfrac{1}{2}\dfrac{d}{dt} \norm{\partial \sigma, \nabla \partial \sigma}{L^2}^2 + \frac{1}{\varepsilon^2}\norm{\partial\sigma}{L^2}^2 = - \int (\partial_t D(\partial u) - \partial_t \Delta D(\partial u) ) : \partial \sigma \,dx \\
        - \int \lbrace \partial u \cdot \nabla \sigma + \partial \lbrack u \cdot \nabla D(u) \rbrack + \partial \lbrack Q(\sigma + D(u), \nabla u) \rbrack  \rbrace : \partial \sigma \,dx \\
        \leq \frac{1}{2 \varepsilon^2} \norm{\partial\sigma}{L^2}^2 + \varepsilon^2 \norm{\substack{\partial_t D(\partial u) - \partial_t \Delta D(\partial u), \partial u\cdot \nabla \sigma, \\ \partial \lbrack u \cdot \nabla D(u) \rbrack, \partial \lbrack Q(\sigma + D(u),\nabla u)\rbrack}}{L^2}^2. 
    \end{gathered}
\end{equation}
Therefore, integrating \eqref{low-ene:006} and \eqref{low-ene:008} in time and summing over all $ \partial \in \lbrace \partial_x, \partial_y \rbrace $ imply
\begin{equation}
    \label{low-ene:009}
    \begin{gathered}
        \sup_{0\leq s \leq t} \norm{\nabla u(s),\omega(s), \nabla \sigma(s), \nabla^2 \sigma
        (s)}{L^2}^2 + \int_0^t \norm{\nabla^2 u(s),\nabla \omega(s), \frac{\nabla\sigma(s)}{\varepsilon}}{L^2}^2\,ds \\ \lesssim \norm{\omega_\mrm{in},\nabla\sigma_\mrm{in}}{L^2}^2 
        + \varepsilon^2 \underbrace{\int_0^t \norm{\substack{\frac{\nabla \sigma}{\varepsilon},\partial_t D(\nabla u) - \partial_t \Delta D(\nabla u), \nabla u\cdot \nabla \sigma, \\ \nabla \lbrack u \cdot \nabla D(u) \rbrack,  \nabla \lbrack Q(\sigma + D(u),\nabla u)\rbrack}}{L^2}^2 \,ds}_{=:\mathcal N_2(t)}, 
    \end{gathered}
\end{equation}
where we have applied the following estimates:
\begin{equation}
\label{low-ene:010}
    \norm{\nabla^{s+1} u}{L^2} \lesssim \norm{\nabla^s \omega}{L^2}, \qquad \text{for} \ \forall s. 
\end{equation}

\subsection*{$H^2$ estimate}
Applying $ \partial \in \lbrace \partial_x, \partial_y \rbrace $ in \eqref{low-ene:005} leads to
\begin{equation}
    \label{low-ene:011}
    \partial_t \partial \omega + u \cdot \nabla  \partial \omega + \partial u \cdot \nabla \omega = \Delta \partial \omega + \nabla \times \nabla \cdot \partial \sigma. 
\end{equation}
Taking the $ L^2 $-inner product of \eqref{low-ene:011} with $ \partial \omega $ and applying integration by parts yield
\begin{equation}
    \label{low-ene:012}
    \begin{gathered}
    \dfrac{1}{2} \dfrac{d}{dt} \norm{\partial \omega}{L^2}^2 + \norm{\nabla \partial \omega}{L^2}^2 = - \int (\partial u \cdot \nabla) \omega \cdot \partial \omega \,dx  - \int  (\nabla\cdot \partial \sigma) \cdot \nabla^\top \partial \omega \,dx \\
    \leq \frac{1}{4} \norm{\nabla \partial \omega}{L^2}^2 + \norm{ \nabla^2 \sigma}{L^2}^2 + \underbrace{\norm{\nabla \omega}{L^4}^2}_{\mathclap{\lesssim \norm{\nabla \omega}{L^2} \norm{\nabla^2 \omega}{L^2}}} \norm{\nabla u}{L^2} \\
    \leq \frac{1}{2} \norm{\nabla^2 \omega}{L^2}^2 + \norm{\nabla^2 \sigma}{L^2}^2 + C \norm{\nabla u}{L^2}^2 \norm{\nabla \omega}{L^2}^2,
    \end{gathered}
\end{equation}
for some constant $ C \in (0,\infty) $, where we have applied the Ladyzhenskaya inequality; see, e.g., \cite{ConstantinFoias1988,Temam1975}.
Similarly, applying $ \partial \in \lbrace \partial_x, \partial_y \rbrace $ to \eqref{low-ene:007} yields
\begin{equation}
    \label{low-ene:013}
    \begin{gathered}
    \partial_t ( \partial^2 \sigma - \Delta \partial^2 \sigma  ) + u\cdot \nabla \partial^2 \sigma + 2 \partial u \cdot \nabla \partial \sigma + \partial^2 u \cdot \nabla \sigma + \frac{1}{\varepsilon^2} \partial^2 \sigma \\
    = - \partial_t D (\partial^2 u) + \partial_t \Delta D(\partial^2 u)  - \partial^2 \lbrack u\cdot \nabla D(u)\rbrack - \partial^2\lbrack Q(\sigma+D(u),\nabla u) \rbrack.
    \end{gathered}
\end{equation}
Taking the $ L^2 $-inner product of \eqref{low-ene:013} with $ \partial^2 \sigma $ and applying integration by parts lead to
\begin{equation}
    \label{low-ene:014}
    \begin{aligned}
        & \dfrac{1}{2} \dfrac{d}{dt} \norm{\partial^2 \sigma, \nabla \partial^2 \sigma}{L^2}^2 + \frac{1}{\varepsilon^2} \norm{\partial^2 \sigma}{L^2}^2 \\ & = - \int \lbrace \substack{\partial_t D (\partial^2 u) - \partial_t \Delta D(\partial^2 u)+2 \partial u \cdot \nabla \partial \sigma + \partial^2 u \cdot \nabla \sigma\\ +\partial^2 \lbrack u\cdot \nabla D(u)\rbrack + \partial^2\lbrack Q(\sigma+D(u),\nabla u) \rbrack} \rbrace : \partial^2 \sigma \,dx \\
        &\leq \frac{1}{2\varepsilon^2}\norm{\partial^2 \sigma}{L^2}^2 + \varepsilon^2 \norm{\substack{\partial_t D (\partial^2 u)- \partial_t \Delta D(\partial^2 u)+2 \partial u \cdot \nabla \partial \sigma + \partial^2 u \cdot \nabla \sigma\\ +\partial^2 \lbrack u\cdot \nabla D(u)\rbrack + \partial^2\lbrack Q(\sigma+D(u),\nabla u) \rbrack}}{L^2}^2. 
    \end{aligned}
\end{equation}
Therefore, integrating \eqref{low-ene:012} and \eqref{low-ene:014} in time and summing over all $ \partial \in \lbrace \partial_x, \partial_y \rbrace $ imply
\begin{equation}
    \label{low-ene:015}
    \begin{gathered}
        \sup_{0\leq s \leq t} \norm{\nabla^2 u(s), \nabla \omega(s), \nabla^2 \sigma(s), \nabla^3 \sigma(s)}{L^2}^2 + \int_0^t \norm{\nabla^3 u(s),\nabla^2 \omega(s), \frac{\nabla^2 \sigma(s)}{\varepsilon}}{}^2\,ds\\
        \lesssim \norm{\nabla \omega_\mrm{in},\nabla^2 \sigma_\mrm{in}}{L^2}^2 + \sup_{0\leq s \leq t}\norm{\nabla u(s)}{L^2}^2 \int_0^t \norm{\nabla \omega(s)}{L^2}^2 \,ds \\
        + \varepsilon^2 \underbrace{\int_0^t \norm{\substack{\frac{\nabla^2 \sigma}{\varepsilon}, \partial_t D (\nabla^2 u) - \partial_t \Delta D(\nabla^2 u),  \nabla u \cdot \nabla \nabla \sigma,  \nabla^2 u \cdot \nabla \sigma,\\ \partial^2 \lbrack u\cdot \nabla D(u)\rbrack,  \nabla^2\lbrack Q(\sigma+D(u),\nabla u) \rbrack }}{L^2}^2 \,ds}_{=:\mathcal N_3(t)} \\
        \lesssim \norm{\nabla \omega_\mrm{in},\nabla^2 \sigma_\mrm{in}}{L^2}^2 + \norm{\omega_\mrm{in}, \nabla \sigma_\mrm{in}}{L^2}^4 + \varepsilon^4 \mathcal N_2^2(t) + \varepsilon^2 \mathcal N_3(t),
    \end{gathered}
\end{equation}
where we have substituted \eqref{low-ene:009}.

\subsection*{Estimates of $ \mathcal N_j $, $ j = 1,2,3 $ and summary of the low norm estimate}

Let the low norm be
\begin{equation}
    \label{def:low-norm}
    \mathcal E_\mrm{low}(t):= \sup_{0\leq s \leq t} \norm{u(s), \sigma(s),\nabla \sigma(s)}{H^2}^2 + \int_0^t \norm{\nabla u(s), \frac{\sigma(s)}{\varepsilon}}{H^2}^2 \,ds,
\end{equation}
and the high norm be
\begin{equation}
    \label{def:high-norm}
    \mathcal E_\mrm{high}(t) := \sup_{0\leq s \leq t} \norm{u (s), \varepsilon \tau (s),\varepsilon \nabla \tau (s)}{H^{6}}^2 + \int_0^t \norm{\tau(s)}{H^{6}}^2\,ds.
\end{equation}
Then the total energy is given by
\begin{equation}
    \label{def:total-norm}
    \mathcal E_\mrm{total}(t) := \mathcal E_\mrm{low}(t) + \mathcal E_\mrm{high}(t). 
\end{equation}

We focus on the estimate of $ \mathcal N_3 $, while the estimates of $ \mathcal N_j $, $ j = 1,2 $, follow similarly. 

Similar to \eqref{low-ene:005}, one can calculate from \subeqref{sys:oldroyd-bb}{1} that
\begin{equation}
\label{n-ene:001}
    \partial_t \omega = - u \cdot \nabla \omega + \nabla\times \nabla \cdot  \tau.
\end{equation}
Then one has that
\begin{equation}
    \label{n-ene:002}
    \begin{gathered}
        \norm{\partial_t D(\nabla^2 u) - \partial_t \Delta D(\nabla^2 u)}{L^2} \lesssim \norm{\partial_t \omega}{H^{4}} \lesssim \norm{u \cdot \nabla \omega}{H^4} + \norm{\tau}{H^6} \\
        \lesssim \norm{u}{H^3}\norm{u}{H^6} + \norm{\tau}{H^6}.
    \end{gathered}
\end{equation}
Similarly, one can calculate that
\begin{equation}
    \label{n-ene:003}
    \norm{\partial^2 \lbrack u\cdot \nabla D(u)\rbrack}{L^2} \lesssim \norm{u}{H^3} \norm{u}{H^6}.
\end{equation}
Furthermore, substituting $ \sigma = \tau - D(u) $ from \eqref{def:tight-variable} yields
\begin{equation}
    \label{n-ene:004}
    \begin{gathered}
        \norm{\nabla u \cdot \nabla \nabla \sigma, \nabla^2 u \cdot \nabla \sigma,  \nabla^2\lbrack Q(\sigma+D(u),\nabla u)\rbrack}{L^2} \\
        \lesssim \norm{\nabla u \cdot \nabla \nabla (\tau - D(u)), \nabla^2 u \cdot \nabla (\tau - D(u)), \nabla^2\lbrack Q(\tau,\nabla u)\rbrack}{L^2}\\
        \lesssim \norm{u}{H^3}^2 + \norm{\tau}{H^2}\norm{u}{H^3}.
    \end{gathered}
\end{equation}
Therefore, one has that
\begin{equation}
    \label{n-ene:005}
    \begin{gathered}
        \mathcal N_3 \lesssim \int_0^t \biggl(\norm{\frac{\sigma(s)}{\varepsilon}}{H^2}^2+ \norm{\tau(s)}{H^6}^2 + \norm{u(s)}{H^3}^2\norm{u(s)}{H^6}^2  + \norm{\tau (s)}{H^2}^2\norm{u(s)}{H^3}^2 \biggr) \,ds \\
        \lesssim \mathcal E_\mrm{total} + \mathcal E_\mrm{total}^2,
    \end{gathered}
\end{equation}
where we have used the Poincar\'e inequality \eqref{poincare}.  Repeating the same arguments for $ \mathcal N_1 $ and $ \mathcal N_2 $ yields
\begin{equation}
    \label{est:nonlinear-low}
    \mathcal N_i(t) \lesssim \mathcal E_\mrm{total}(t) + \mathcal E_\mrm{total}^2(t). 
\end{equation}
Therefore, one can conclude from \eqref{low-ene:003}, \eqref{low-ene:009}, and \eqref{low-ene:015} that
\begin{equation}
    \label{low-ene:total}
    \mathcal E_\mrm{low}(t) \leq \mathfrak c_1 (\mathcal E_\mrm{low}(0) + 1)^2 + \mathfrak c_1 \varepsilon^2 (\mathcal E_\mrm{total}(t) + 1)^4,
\end{equation}
for some constant $ \mathfrak c_1 \in (0,\infty) $. 

\subsection*{Decay estimate}
We conclude this section by preparing the decay estimate of the low norm. 

From \eqref{low-ene:001} and \eqref{low-ene:006}, 
one has that
\begin{equation}
    \label{low-decay:001}
    \begin{gathered}
        \dfrac{d}{dt} \norm{u, \omega}{L^2}^2 + \norm{\nabla u, \nabla\omega}{L^2}^2 \lesssim \varepsilon^2 \norm{{\frac{\sigma}{\varepsilon}}}{H^1}^2. 
    \end{gathered}
\end{equation}
Applying the Poincar\'e inequality, one can conclude from \eqref{low-decay:001} that there exists $ \mathfrak d_1 \in (0,\infty) $ such that
\begin{equation}
    \label{low-decay:002}
    \dfrac{d}{dt} \norm{u, \omega}{L^2}^2 + \mathfrak d_1 \norm{u, \omega}{L^2}^2 \lesssim \varepsilon^2 \norm{{\frac{\sigma}{\varepsilon}}}{H^1}^2.
\end{equation}
Then direct calculation yields that
\begin{equation}
    \label{low-decay:003}
    \begin{gathered}
    \norm{u(t)}{H^1}^2 \lesssim e^{-\mathfrak d_1 t} \norm{u_\mrm{in}}{H^1}^2 + \varepsilon^2 \int_0^t e^{\mathfrak d_1 (s-t)} \norm{\frac{\sigma(s)}{\varepsilon}}{H^1}^2 \,ds \\
    \lesssim e^{-\mathfrak d_1 t} \norm{u_\mrm{in}}{H^1}^2 + \varepsilon^2 \mathcal E_\mrm{low}(t).
    \end{gathered}
\end{equation}

\section{High norm estimate for the relaxed variables}
\label{sec:high-norm-est}

\subsection*{$ H^6 $ estimate}
Now we return to the relaxed system \eqref{sys:oldroyd-bb}. Let $ \partial \in \lbrace \partial_x, \partial_y \rbrace $. To shorten our presentation, we hereafter omit the combinatorial constants arising from the Leibniz rule. Applying $ \partial^6 $ to system \eqref{sys:oldroyd-bb} yields
\begin{equation}
    \label{sys:high-norm}
    \begin{cases}
        \partial_t \partial^6 u + u \cdot \nabla \partial^6 u + \nabla \partial^6 p - \nabla \cdot \partial^6 \tau = - \sum_{j=1}^{6}\partial^j u \cdot \nabla \partial^{6-j} u, \\
        \partial_t ( \partial^6 \tau - \Delta \partial^6 \tau) + u \cdot \nabla \partial^6 \tau 
        - \frac{1}{\varepsilon^2} (D(\partial^6 u)  - \partial^6 \tau) \\
        \qquad\qquad = - \sum_{j=1}^6 \partial^j u \cdot \nabla \partial^{6-j} \tau - \partial^6 Q(\tau, \nabla u),\\
        \nabla \cdot \partial^6 u = 0.
    \end{cases}
\end{equation}

Taking the $ L^2 $-inner product of \subeqref{sys:high-norm}{1} with $ \partial^6 u $ and \subeqref{sys:high-norm}{2} with $ \varepsilon^2  \partial^6 \tau $, respectively, summing all resulting equations, and applying integration by parts lead to
\begin{equation}
    \label{h-ene:001}
    \begin{aligned}
        & \frac{1}{2}\dfrac{d}{dt} \norm{\partial^6 u, \varepsilon \partial^6 \tau, \varepsilon \nabla \partial^6 \tau}{L^2}^2 + \norm{\partial^6 \tau}{L^2}^2 \\
        & = - \sum_{j=1}^6 \int \partial^j u \cdot \nabla \partial^{6-j} u \cdot \partial^6 u  \, dx 
        - \sum_{j=1}^6 \int  \varepsilon^2 \partial^j u \cdot \nabla \partial^{6-j} \tau : \partial^6 \tau \,dx \\
        & \qquad - \int \varepsilon^2 \partial^6 Q(\tau, \nabla u): \partial^6 \tau \,dx =: J_1 + J_2 + J_3. 
    \end{aligned}
\end{equation}
To calculate $ J_1 $, one can write $ J_1 $ as
\begin{equation}
    \label{h-ene:002}
    \begin{aligned}
    J_1 & = - \biggl( \sum_{j=1}^{3} + \sum_{j=4}^{6}\biggr) \int \partial^j u \cdot \nabla \partial^{6-j} u \cdot \partial^6 u\, dx 
    \lesssim \int \vert \nabla u \vert \vert  \nabla^{6} u \vert \vert \nabla^6 u \vert \,dx \\
    & \qquad +\int \vert \nabla^2 u \vert \vert  \nabla^{5} u \vert \vert \nabla^6 u \vert \,dx + \int \vert \nabla^3 u \vert \vert  \nabla^{4} u \vert \vert \nabla^6 u \vert \,dx\\
    & \lesssim \underbrace{\norm{\nabla u}{L^\infty}}_{\mathclap{\lesssim \norm{u}{H^1}^{1/2} \norm{u}{H^3}^{1/2}}} \norm{\nabla^6 u}{L^2}^2 + \underbrace{\norm{\nabla^2 u}{L^\infty}}_{{\lesssim \norm{u}{H^1}^{1/2}\norm{u}{H^3}^{1/6}\norm{u}{H^6}^{1/3}}} \underbrace{\norm{\nabla^5 u}{L^2}}_{{\norm{u}{H^3}^{1/3} \norm{u}{H^6}^{2/3}}} \norm{\nabla^6 u}{L^2} \\
    & \qquad + \underbrace{\norm{\nabla^3 u}{L^2}}_{\norm{u}{H^1}^{1/2} \norm{u}{H^3}^{1/6} \norm{u}{H^6}^{1/3}} \underbrace{\norm{\nabla^4 u}{L^\infty}}_{\norm{u}{H^3}^{1/3} \norm{u}{H^6}^{2/3}} \norm{\nabla^6 u}{L^2}  \lesssim  \norm{u}{H^1}^{1/2} \norm{u}{H^3}^{1/2} \mathcal E_\mrm{high}, 
    \end{aligned} 
\end{equation}
where we have applied the Gagliardo–Nirenberg interpolation. 
To calculate $ J_2 $, one can write $ J_2 $ as
\begin{equation}
    \label{h-ene:003}
    \begin{aligned}
        J_2 & = - \biggl( \sum_{j=1}^{3} + \sum_{j=4}^{6}\biggr) \int \varepsilon^2 \partial^j u \cdot \nabla \partial^{6-j} \tau : \partial^6 \tau \, dx =: J_{2,1} + J_{2,2}. 
    \end{aligned}
\end{equation}
Similarly as in \eqref{h-ene:002}, one can verify that
\begin{equation}
    \label{h-ene:004}
    \begin{gathered}
    J_{2,1} \lesssim \int \vert \partial u \vert \vert \varepsilon \nabla\partial^5 \tau\vert \vert \varepsilon \partial^6 \tau \vert \,dx + \int \vert \partial^2 u \vert \vert \varepsilon \nabla\partial^4 \tau\vert \vert \varepsilon \partial^6 \tau \vert \,dx \\
    + \int \vert \partial^3 u \vert \vert \varepsilon \nabla\partial^3 \tau\vert \vert \varepsilon \partial^6 \tau \vert \,dx \\
    \lesssim  \norm{u}{H^1}^{1/2} (\norm{u}{H^3}^{1/2} + \norm{u}{H^3}^{1/6} \norm{\varepsilon \tau}{H^3}^{1/3})    \mathcal E_\mrm{high} \\
    \lesssim \norm{u}{H^1}^{1/2} (\norm{u}{H^3}^{1/2} + \varepsilon^{1/2} \norm{\tau}{H^6}^{1/2}) \mathcal E_\mrm{high}.
    \end{gathered}
\end{equation}
Meanwhile, for $J_{2,2}$, we have
\begin{equation}
    \label{h-ene:005}
    \begin{aligned}
    J_{2,2} & = - \int \varepsilon^2 \partial^4 u \cdot \nabla \partial^{2} \tau : \partial^6 \tau \, dx - \int \varepsilon^2 \partial^5 u \cdot \nabla \partial \tau : \partial^6 \tau \, dx \\
    & \qquad - \int \varepsilon^2 \partial^6 u \cdot \nabla  \tau : \partial^6 \tau \, dx\\
    &\overset{\mathclap{\text{integration}}}{\underset{\mathclap{\text{by parts}}}{=}} \quad \int \varepsilon^2 \partial^5 u \cdot \nabla \partial \tau : \partial^6 \tau \, dx + \int \varepsilon^2 \partial^4 u \cdot \nabla \partial \tau : \partial^7 \tau \, dx \\
    & \qquad - \int \varepsilon^2 \partial^5 u \cdot \nabla \partial \tau : \partial^6 \tau \, dx  - \int \varepsilon^2 \partial^6 u \cdot \nabla  \tau : \partial^6 \tau \, dx \\
    & \lesssim \varepsilon \int \vert \nabla^5 u \vert \vert   \nabla^2 \tau \vert \vert \varepsilon \nabla^6 \tau \vert \,dx + \varepsilon \int \vert \nabla^4 u \vert \vert   \nabla^2 \tau \vert \vert \varepsilon \nabla^7 \tau \vert \,dx \\
    & \qquad
     + \varepsilon \int \vert \nabla^6 u \vert \vert   \nabla \tau \vert \vert \varepsilon \nabla^6 \tau \vert \,dx 
    \lesssim \varepsilon \norm{\tau}{H^2} \mathcal E_\mrm{high}.
    \end{aligned}
\end{equation}
Thus we have that
\begin{equation}
    \label{h-ene:005}
    J_2 \lesssim \norm{u}{H^1}^{1/2} (\norm{u}{H^3}^{1/2} + \varepsilon^{1/2} \norm{\tau}{H^6}^{1/2})   \mathcal E_\mrm{high} + \varepsilon \norm{\tau}{H^2} \mathcal E_\mrm{high}. 
\end{equation}
To calculate $ J_3 $, since $ Q(\cdot, \cdot) $ is bilinear, one can write $ J_3 $ as
\begin{equation}
    \label{h-ene:006}
    \begin{aligned}
        J_3 & = - \sum_{j=1}^6 \int \varepsilon^2 Q(\partial^j \tau, \nabla \partial^{6-j} u):\partial^6 \tau \,dx  - \int \varepsilon^2 Q(\tau, \nabla \partial^6 u) : \partial^6 \tau \,dx \\
        & =: J_{3,1} + J_{3,2}. 
    \end{aligned}
\end{equation}
The estimate of $ J_{3,1} $ follows the same arguments as for the $J_2$ estimate, i.e.,
\begin{equation}
    \label{h-ene:007}
    J_{3,1} \lesssim \norm{u}{H^1}^{1/2} (\norm{u}{H^3}^{1/2} + \varepsilon^{1/2} \norm{\tau}{H^6}^{1/2})  \mathcal E_\mrm{high} + \varepsilon \norm{\tau}{H^2} \mathcal E_\mrm{high}. 
\end{equation}
To estimate $ J_{3,2} $, applying integration by parts yields that
\begin{equation}
    \label{h-ene:008}
    \begin{aligned}
    J_{3,2} & = \int \varepsilon^2 Q(\partial \tau, \nabla \partial^5 u) : \partial^6 \tau \,dx + \int \varepsilon^2 Q( \tau, \nabla \partial^5 u) : \partial^7 \tau \,dx\\
    & \lesssim \varepsilon \norm{\nabla \tau}{L^4} \norm{\nabla^6 u}{L^2} \norm{\varepsilon \nabla^6 \tau}{L^4} + \varepsilon \norm{\tau}{L^\infty} \norm{\nabla^6 u}{L^2} \norm{\varepsilon \nabla^7 \tau}{L^2}\\
    & \lesssim \varepsilon \norm{\tau}{H^2} \mathcal E_\mrm{high}.
    \end{aligned}
\end{equation}
Therefore, one has that
\begin{equation}
    \label{h-ene:009}
    J_3 \lesssim \norm{u}{H^1}^{1/2} (\norm{u}{H^3}^{1/2} + \varepsilon^{1/2} \norm{\tau}{H^6}^{1/2})  \mathcal E_\mrm{high} + \varepsilon \norm{\tau}{H^2} \mathcal E_\mrm{high}.
\end{equation}

\smallskip

To conclude this section, the same estimates as in \eqref{h-ene:001}--\eqref{h-ene:009} hold for all spatial derivatives no larger than $ 6 $. Therefore, one has that
\begin{equation}
    \label{h-ene:010}
    \frac{1}{2}\dfrac{d}{dt} \norm{u, \varepsilon \tau, \varepsilon \nabla \tau}{H^6}^2 + \norm{\tau}{H^6}^2 \lesssim \norm{u}{H^1}^{1/2} (\norm{u}{H^3}^{1/2} + \varepsilon^{1/2} \norm{\tau}{H^6}^{1/2})  \mathcal E_\mrm{high} + \varepsilon \norm{\sigma, \nabla u}{H^2} \mathcal E_\mrm{high}.
\end{equation}
where we have used the fact, thanks to \eqref{def:tight-variable},
\begin{equation}
    \label{h-ene:011}
    \norm{\tau}{H^2} \lesssim \norm{\sigma, \nabla u}{H^2}. 
\end{equation}
Recalling the definition of high norm in \eqref{def:high-norm}, one can conclude from \eqref{h-ene:010}, after applying Gr\"onwall's inequality, that
\begin{equation}
    \label{h-ene:total}
    \mathcal E_\mrm{high}(t) \leq \mathcal E_\mrm{high}(0) e^{\mathfrak c_2 \int_0^t (\norm{u(s)}{H^1}^{1/2} (\norm{u}{H^3}^{1/2} + \varepsilon^{1/2} \norm{\tau}{H^6}^{1/2})  + \varepsilon \norm{\nabla u(s),\frac{\sigma(s)}{\varepsilon}}{H^2}) \,ds },
\end{equation}
for some constant $ \mathfrak c_2 \in(0,\infty) $. Meanwhile, using \eqref{low-decay:003}, \eqref{low-ene:total}

\begin{equation}
    \label{h-ene:total-001}
    \begin{gathered}
        \int_0^t (\norm{u(s)}{H^1}^{1/2} (\norm{u}{H^3}^{1/2} + \varepsilon^{1/2} \norm{\tau}{H^6}^{1/2})  + \varepsilon \norm{\nabla u(s),\frac{\sigma(s)}{\varepsilon}}{H^2}) \,ds \\
        \lesssim \int_0^t \lbrack e^{-\mathfrak d_1 s/4} \mathcal E_\mrm{low}^{1/4}(0) + \varepsilon^{1/2} \mathcal E_\mrm{low}^{1/4}(s) \rbrack (\norm{u}{H^3}^{1/2} + \varepsilon^{1/2} \norm{\tau}{H^6}^{1/2}) \,ds + \varepsilon t^{1/2} \mathcal E_\mrm{low}^{1/2}(t) \\
        \lesssim \mathcal E_\mrm{low}^{1/4}(0) (\mathcal E_\mrm{low}^{1/4}(t) + \varepsilon^{1/2} \mathcal E_\mrm{high}(t)) + \varepsilon^{1/2} t^{3/4} (\mathcal E_\mrm{low}^{1/2}(t) + \varepsilon \mathcal E_\mrm{high}^{1/2}(t)) + \varepsilon t^{1/2} \mathcal E_\mrm{low}^{1/2}(t)\\
        \lesssim (\mathcal E_\mrm{low}(0) + 1)^{3/4} + \varepsilon^{1/2} \mathcal E_\mrm{low}^{1/4}(0) (\mathcal E_\mrm{total}(t) + 1) + \varepsilon^{1/2} \mathcal E_\mrm{low}^{1/4}(0) \mathcal E_\mrm{total}^{1/4}(t)\\
        + (\varepsilon^{1/2} t^{3/4} + \varepsilon t^{1/2}) \lbrack (\mathcal E_\mrm{low}(0)+1)+ \varepsilon (\mathcal E_\mrm{total}(t) + 1)^2 + \varepsilon \mathcal E_\mrm{total}^{1/2}(t)\rbrack.
    \end{gathered}
\end{equation}

\section{Long time existence}
\label{sec:long-time}
In this section, we give the proof of Theorem \ref{thm:long-time}.
\begin{proof}[Proof of Theorem \ref{thm:long-time}]
    It follows from \eqref{low-ene:total}, \eqref{h-ene:total}, and \eqref{h-ene:total-001} that, one has
\begin{equation}
    \label{time-ene:001}
    \mathfrak F \leq \mathfrak c_3 ( 1 + \varepsilon^2 \mathfrak F^4) + \mathfrak c_3 e^{\mathfrak c_3 \lbrack 1 + \varepsilon^{1/2} \mathfrak F + (\varepsilon^{1/2} t^{3/4} + \varepsilon t^{1/2})(1 + \varepsilon \mathfrak F^2)}
\end{equation}
for some constant $ \mathfrak c_3 \in (0,\infty) $, depending only on the initial data $ \mathcal E_\mrm{total}(0) $, where 
\begin{equation}
    \label{time-ene:002}
    \mathfrak F = \mathfrak F(t) := \mathcal E_\mrm{total}(t) + 1. 
\end{equation}
Then one can conclude from \eqref{time-ene:001} that for some $ 0< \varepsilon_0 \ll 1 $ and $ \mathfrak c_4 \in (0,\infty) $, provided that
\begin{equation}
    \label{time-ene:003}
    0 < \varepsilon \leq \varepsilon_0,
\end{equation}
one has that 
\begin{equation}
\label{time-ene:004}
    \mathfrak F (t) \leq 2 \mathfrak c_3 + \mathfrak c_3 e^{4 \mathfrak c_3} \qquad \text{for} \ 0 \leq t \leq \frac{1}{\varepsilon^{2/3}}.
\end{equation}
This finishes the proof of theorem \ref{thm:long-time}.
\end{proof}

\section{High Weissenberg number imit}
\label{sec:hw-limit}
In this section, we prove Theorem \ref{thm:hw-limit}.

\begin{proof}
    For arbitrary $ T \in (0,\infty) $, let $ \varepsilon_T := \min\lbrace \varepsilon_0, T^{-3/2} \rbrace $. Then for $ \varepsilon \leq \varepsilon_T \leq \varepsilon_0 $, one has that from \eqref{thm-1-est} (or equivalently, \eqref{time-ene:004})
\begin{equation}
    \label{hw-limit:001}
    \begin{gathered}
    \sup_{0 \leq t \leq T} (\norm{u(t),\varepsilon \tau(t), \varepsilon\nabla\tau(s)}{H^6}^2 + \norm{\tau(t) - D(u)(t), \nabla (\tau(t) - D(u)(t))}{H^2}^2 ) \\
    + \int_0^{T} (\norm{\tau(t)}{H^6}^2 + \norm{\nabla u(t), \frac{\tau(t) - D(u)(t)}{\varepsilon}}{H^2}^2 ) \,dt < C,
    \end{gathered}
\end{equation}
for some constant $ C \in (0,\infty) $. In particular, one has that, as $ \varepsilon \rightarrow 0 $, 
\begin{align}
        u & \overset{*}{\rightharpoonup} v && \text{in} \ L^\infty(0,T;H^6),\\
        u & \rightharpoonup v && \text{in} \ L^2(0,T;H^3), \\
    \label{hw-limit:003}
        \tau & \rightarrow D(v) && \text{in} \ L^2(0,T;H^2),         
\end{align}
for some $ v \in L^\infty(0,T;H^6) $.
Meanwhile, from \subeqref{sys:oldroyd-bb}{1}, one has that
\begin{equation}
    \partial_t u = - u \cdot \nabla u - \nabla p + \Delta u + \nabla \cdot (\tau - D(u)), 
\end{equation}
and thus, thanks to \eqref{hw-limit:001}, one can conclude that
\begin{equation}
    \label{hw-limit:002}
    \sup_{0\leq t \leq T}\norm{\partial_t u(t)}{H^2}^2 \lesssim \sup_{0\leq t \leq T} \norm{u(t)}{H^3}^2 + \norm{u(t)}{H^4} + \norm{\nabla (\tau(t) - D(u)(t))}{H^2}^2  < \infty. 
\end{equation}
Therefore, applying the Aubin--Lions compactness lemma (see, e.g., \cite{Simon2003}) leads to
\begin{equation}
    \label{hw-limit:004}
    \begin{aligned}
        u & \rightarrow  v && \text{in} \ C(0,T;H^5). \\
    \end{aligned}
\end{equation}
In particular, with \eqref{hw-limit:003} and \eqref{hw-limit:004}, it is easy to verify that $ v $ solves the Navier-Stokes equations, i.e., 
\begin{equation}
    \label{hw-limit:005}
    \partial_t v + v \cdot \nabla v + \nabla \pi = \nabla \cdot D(v), \qquad \nabla \cdot v = 0. 
\end{equation}
This finishes the proof of theorem \ref{thm:hw-limit}.
\end{proof}

\section*{Acknowledgements}
WW was partially supported by the Simons Foundation Travel Support for Mathematicians (No. 0007730).

\end{document}